\documentclass[10pt,reqno]{amsart}

\usepackage{amsmath,amssymb}
\usepackage{a4wide}
\usepackage[all]{xy}
\usepackage{mathbbol}

\newtheorem{cor}{Corollary}
\newtheorem{thm}{Theorem}

\newcommand{\cst}{\ifmmode{\mathrm{C}^*}\else{$\mathrm{C}^*$}\fi}
\newcommand{\tens}{\otimes}
\newcommand{\eps}{\varepsilon}
\newcommand{\ph}{\varphi}
\newcommand{\comp}{\circ}
\newcommand{\RR}{\mathbb{R}}
\newcommand{\TT}{\mathbb{T}}
\newcommand{\CC}{\mathbb{C}}
\newcommand{\sS}{\mathbb{S}}
\newcommand{\GG}{\mathbb{G}}
\newcommand{\ZZ}{\mathbb{Z}}
\newcommand{\cI}{\mathcal{I}}
\newcommand{\cJ}{\mathcal{J}}
\newcommand{\cK}{\mathcal{K}}
\newcommand{\cA}{\mathcal{A}}
\newcommand{\cR}{\mathcal{R}}
\newcommand{\bz}{\boldsymbol{z}}
\newcommand{\bx}{\boldsymbol{x}}
\newcommand{\bs}{\boldsymbol{s}}
\newcommand{\id}{\mathrm{id}}
\newcommand{\I}{\mathbb{1}}
\renewcommand{\Bar}[1]{\overline{#1}}
\newcommand{\st}{\:\vline\:}
\DeclareMathOperator{\C}{C}

\begin{document}

\title{Non existence of group structure on some quantum spaces}

\date{\today}

\author{Piotr M.~So{\l}tan}
\address{Department of Mathematical Methods in Physics\\
Faculty of Physics\\
Warsaw University}
\email{piotr.soltan@fuw.edu.pl}

\thanks{Research partially supported by Polish government grant
no.~N201 1770 33.}

\begin{abstract}
We prove that some well known compact quantum spaces like quantum tori and some quantum two-spheres do not admit a compact quantum group structure. This is achieved by considering existence of traces, characters and nuclearity of the corresponding $\mathrm{C}^*$-algebras.
\end{abstract}

\maketitle

\section{Introduction}

Classical topology gives us many tools with which one can investigate whether a given topological space carries a structure of topological group. In this paper we want to consider the same problem for non-commutative or \emph{quantum} spaces. Not surprisingly some standard tricks from classical topology are not available in the non-commutative setting. As an example consider the standard exercise to show that an interval with one or two end-points does not admit a structure of a topological group. The solution uses the fact that the neighborhoods of an end-point cannot be carried over homeomorphically onto neighborhoods of interior points of the interval and the hypothetical group law would provide such homeomorphisms. However, as for now, such arguments cannot be carried over to non-commutative topology.

Interestingly, however, many new tools which have no classical analog can be employed to study this problem for quantum spaces. Those used in this paper include studying existence of traces on considered \cst-algebras, the structure of the space of their characters and also properties such as nuclearity.

We will only consider \emph{compact quantum spaces,} i.e.~ones described by unital \cst-algebras via the well established correspondence extending the Gelfand equivalence (\cite{ncgnt,ncg}). The reason for this if twofold. First of all there is still some doubt as to a precise definition of a non-compact quantum group (quantum space with group structure) although the Kustermans-Vaes definition (\cite{kv}) has much more appeal than the definition involving modular multiplicative unitaries (\cite{mmu2}). The second reason is the very rich structure compact quantum groups have and ample literature on the subject (see e.g.~\cite{cqg,coam}).

Before proceeding let us address a related problem. Recall that a \emph{compact quantum semigroup} is a pair $\sS=(A,\Delta)$ consisting of a unital \cst-algebra $A$ and a \emph{comultiplication}, i.e.~a coassociative morphism $\Delta:A\to{A\tens{A}}$ (\cite{coam,qs}). A \emph{compact quantum group} is a compact quantum semigroup $\GG=(A,\Delta)$ such that the sets
\[
\bigl\{\Delta(a)(\I\tens{b})\st{a,b\in{A}}\bigr\}\quad\text{and}\quad
\bigl\{(a\tens\I)\Delta(b)\st{a,b\in{A}}\bigr\}
\]
are linearly dense in $A\tens{A}$ (\cite[Definition 2.1]{cqg}). The problem whether a given quantum semigroup is in fact a compact quantum group might be quite difficult to approach directly. In \cite{qcom} we showed that if $M$ is a finite dimensional \cst-algebra then the quantum semigroup $\sS=(A,\Delta)$ of all self-maps of the finite quantum space underlying $M$ is not a compact quantum group unless $M=\CC$ (\cite[Proposition 2.1]{qcom}). One must be careful here though. The simplest example obtained by taking $M=\CC^2$ leads to $A=\cst(\ZZ_2*\ZZ_2)$ which, of course, does admit a compact quantum group structure, but with a comultiplication different from $\Delta$.

A more explicit example was given in \cite[Corollary 4.4]{qcom} where we showed that the quantum commutant (\cite{qs,qcom}) of a certain automorphism of $M_2$ is not a compact quantum group. In this example $A$ is the universal \cst-algebra generated by three elements $\alpha,\beta$ and $\gamma$ such that $\beta=\beta^*$, $\gamma=\gamma^*$ and
\[
\begin{aligned}
\alpha^*\beta+\gamma\alpha^*+\alpha\gamma+\beta\alpha&=0,&\alpha^2+\beta\gamma&=0,\\
\alpha\beta+\beta\alpha^*&=0,&\gamma\alpha+\alpha^*\gamma&=0,\\
\alpha^*\alpha+\gamma^2+\alpha\alpha^*+\beta^2&=\I,\\
\end{aligned}
\]
while comultiplication $\Delta$ acts on generators in the following way:
\[
\begin{split}
\Delta(\alpha)&=\I\tens\alpha+(\alpha^*\alpha+\gamma^2)\tens(\alpha^*-\alpha)+
\alpha\tens\beta+\alpha^*\tens\gamma,\\
\Delta(\beta)&=(\alpha\gamma+\beta\alpha)\tens(\alpha-\alpha^*)+\beta\tens\beta+\gamma\tens\gamma,\\
\Delta(\gamma)&=(\beta\alpha+\alpha\gamma)\tens(\alpha^*-\alpha)+\gamma\tens\beta+\beta\tens\gamma.
\end{split}
\]
The density conditions from the definition of a compact quantum group are difficult to check. However, using some advanced elements of the theory of compact quantum groups, we are able to prove that $\sS=(A,\Delta)$ is not a compact quantum group.

In a similar context in \cite{so3} we showed that the quantum semigroup of all maps of the quantum space underlying $M_2$ preserving a certain degenerate state is not a compact quantum group. In this case $A$ is the universal \cst-algebra generated by two elements $\beta$ and $\delta$ satisfying the following relations
\[
\begin{aligned}
\beta\beta^*&=\I,\!\!\!\!&\delta^2&=0,\\
\beta\delta&=0,\!\!\!\!&\beta\delta^*&=0,\\
\beta^*\beta&+\delta^*\delta\!\!\!\!&+\:\delta\delta^*&=\I
\end{aligned}
\]
and the comultiplication $\Delta$ acts on generators in the following way:
\[
\begin{split}
\Delta(\beta)&=\beta\tens\beta,\\
\Delta(\delta)&=\delta\tens\beta+\beta^*\beta\tens\delta+\delta^*\delta\tens\delta.
\end{split}
\]
It is shown in \cite[Proposition 5.5]{so3} that $\sS=(A,\Delta)$ is not a quantum group.

In contrast to the cases dealt with above, in what follows we will not make any assumptions on the particular form of the comultiplication. We will show that the following compact quantum spaces do not admit \emph{any} compact quantum group structure:
\begin{itemize}
\item all quantum tori,
\item standard Podle\'s sphere,
\item quantum two-spheres of Bratteli-Elliott-Evans-Kishimoto,
\item Natsume-Olsen quantum spheres.
\end{itemize}
The quantum tori have been considered as non-commutative geometric objects in \cite{connes80,rie2} and numerous other papers. For the survey and references on the remaining quantum spaces we recommend \cite{ld} as well as the original papers \cite{podles,beek1,beek2,beek3,natsume}.

Let us now describe briefly the contents of the paper. In the next section we recall some basic elements and terminology of the theory of compact quantum groups. In Section \ref{non} we formulate and prove the main results. This section is split into two subsections. The first dealing with algebras with abundant traces while the other uses the fact that the considered \cst-algebras have very few \emph{classical points}.

\section{Preliminaries on compact quantum groups}\label{pre}

Let $\GG=(A,\Delta)$ be a compact quantum group. The \cst-algebra $A$ is referred to as the \emph{algebra of continuous functions on $\GG$.} Theorem 2.2 of \cite{cqg} tells us that there is a dense unital $*$-subalgebra $\cA$ in $A$ such that $\bigl.\Delta\bigr|_{\cA}:\cA\to\cA\tens_{\mathrm{alg}}\cA$ and with this  comultiplication $\cA$ is a Hopf $*$-algebra. By \cite[Theorem 5.1]{coam} this Hopf $*$-algebra is unique. The enveloping \cst-algebra $A_u$ of $\cA$ carries a comultiplication $\Delta_u:A_u\to{A_u\tens{A_u}}$ such that $\GG_u=(A_u,\Delta_u)$ is a compact quantum group. Note that $\cA$ is then naturally a dense subalgebra of $A_u$.

There is a unique epimorphism $\rho:A_u\to{A}$ extending the identity map on $\cA$ and it satisfies $(\rho\tens\rho)\comp\Delta_u=\Delta\comp\rho$. The compact quantum group $\GG_u$ is called the \emph{universal version} of $\GG$. We will often use the fact that any $*$-character of $\cA$ extends continuously to $A_u$ (\cite[Theorem 3.6]{coam}).

The famous theorem of Woronowicz (\cite[Theorem 4.2]{pseudogr}, \cite[Theorem 2.3]{cqg}) says that if $\GG=(A,\Delta)$ is a compact quantum group then there exists a unique state $h$ on $A$ such that
\[
(h\tens\id)\Delta(a)=(\id\tens{h})\Delta(a)=h(a)\I
\]
for any $a\in{A}$. This state is called the \emph{Haar measure} of $\GG$. The state $h$ might not be faithful, but it is known (\cite[Page 656]{pseudogr}, see also \cite[Theorem 2.1]{coam}) that the ideal $\cJ=\bigl\{a\in{A}\st{h}(a^*a)=0\bigr\}$ is two-sided. Let $A_r=A/\cJ$ and denote by $\lambda$ the quotient map $A\to{A_r}$. The map $\lambda$ is injective on the dense subalgebra $\cA$, so $A_r$ might be viewed as a different completion of $\cA$. It can be shown that there is a comultiplication $\Delta_r:A_r\to{A_r\tens{A_R}}$ extending that of $\cA$ such that $\GG_r=(A_r,\Delta_r)$ is a compact quantum group called the \emph{reduced version} of $\GG$. We have $(\lambda\tens\lambda)\comp\Delta=\Delta_r\comp\lambda$. Moreover the Haar measure $h_r$ of $\GG_r$ is faithful and $h=h_r\comp\lambda$ (similarly the Haar measure $h_u$ of $\GG_u$ satisfies $h_u=h\comp\rho$).

To illustrate the concepts dealt with above take $A$ to be a \cst-completion of the group algebra $\CC[\Gamma]$ of a discrete group $\Gamma$ such that the comultiplication $\gamma\mapsto\gamma\tens\gamma$ extends to a $*$-homomorphism $\Delta:A\to{A\tens{A}}$. Then $\GG=(A,\Delta)$ is a compact quantum group. The $\cst$-algebras $A_u$ and $A_r$ are $\cst(\Gamma)$ and $\cst_r(\Gamma)$. The Haar measure on $\GG_r$ is the von Neumann trace on $\cst_r(\Gamma)$. Clearly if $\Gamma$ is amenable then $A$, $A_u$ and $A_r$ are canonically isomorphic. Note that the example $\GG=(A,\Delta)$ with $A=\cst_r(\mathbb{F}_2)$ shows that the lagebra of continuous functions on a comapct quantum group may be siple.

Now let $\GG=(A,\Delta)$ be a compact quantum group. The special situation when $\rho$ and $\lambda$ are isomorphism is called \emph{co-amenability} of $\GG$. There are various equivalent definitions of this concept and we will use some of them. For details we refer to \cite{coam}. In particular if $\GG$ is co-amenable then the co-unit $e$ of $\cA$ extends to a character of $A$ (because $A=A_u$) and this extension (also denoted by $e$) still satisfies $(e\tens\id)\comp\Delta=(\id\tens{e})\comp\Delta=\id$. One can easily see that if any of the \cst-algebras $A$, $A_u$ or $A_r$ is commutative then they are all the same and $\GG$ is co-amenable.

The Haar measure $h$ of $\GG$ need not be a trace. Whenever $h$ is a trace we say that $\GG$ is of \emph{Kac type}. Groups of Kac type can be characterized in many different ways (see e.g.~\cite[Theorem 2.5]{cqg}). The study of compact quantum group which are not of Kac type lead to the discovery that there always exists a certain family $(f_z)_z\in\CC$ of multiplicative functionals on $\cA$ which encodes modular properties of $h$. We always have $f_0=e$ (the co-unit) and the family is trivial (i.e.~$f_z=f_0$ for all $z$) if and only if $\GG$ is of Kac type (\cite[]{cqg}). Moreover this family is holomorphic in the sense that for any $a\in\cA$ the function $z\mapsto{f_z(a)}$ is entire. For $z$ on the imaginary axis $f_z$ are $*$-caracters, and consequently they extend to characters of $A_u$. We call these functionals the \emph{Woronowicz characters} of $\GG$.


\section{Non existence of compact quantum group structure}\label{non}

\subsection{Algebras with abundant traces}\label{traces}

We will first recall a result formulated in \cite[Remark A.2]{qbc}.

\begin{thm}
Let $\GG=(A,\Delta)$ be a compact quantum group. Assume that for any $a\in{A}$ there exists a tracial state $\tau$ on $A$ such that $\tau(a^*a)\neq{0}$. Then $\GG$ is of Kac type.
\end{thm}

In the proof of the next corollary we will use the known facts about rotation algebras $A_\theta$ (\cite{rie1,rie2}).

\begin{cor}
For any $\theta\in]0,1[$ the rotation algebra $A_\theta$ does not admit a compact quantum group structure.
\end{cor}

\begin{proof}
The \cst-algebra $A_\theta$ admits a faithful trace (unique one if $\theta$ is irrational). Therefore if there existed a comultiplication $\Delta:A_\theta\to{A_\theta}\tens{A_\theta}$ such that $\GG=(A_\theta,\Delta)$ were a compact quantum group, the Haar measure on $\GG$ would be a trace. It is well known that $A_\theta$ is a nuclear \cst-algebra (e.g.~because it is a crossed product of a commutative algebra by an action of an amenable -- in fact commutative -- group). By \cite[Theorem 1.1]{amII} (see also \cite{ruan,reiji}) the compact quantum group $\GG$ must be co-amenable. In particular the \cst-algebra $A_\theta$ would have to admit a continuous co-unit. Clearly this is not the case, since for $\theta$ irrational $A_\theta$ is simple and for rational $\theta$ we have $A_\theta\cong{M_N}\tens{\C(\TT)}$ which clearly does not admit a character.
\end{proof}

It is interesting to note that if $\theta$ is irrational the von Neumann algebra we obtain from $A_\theta$ via the GNS construction is the hyperfinite factor of type $\mathrm{II}_1$. This von Neumann algebra \emph{is} a von Neumann algebra of measurable functions on a compact quantum group. For example it is the group von Neumann algebra of any discrete amenable i.c.c.~group. Moreover, as shown in \cite{dt} the quantum two-torus can be a ``part'' of a compact quantum group (namely the \emph{quantum double torus}).

Let us summarize the argument used to prove that quantum $2$-tori are not quantum groups in the following statement:

\begin{thm}\label{pierw}
Let $A$ be a nuclear unital \cst-algebra having a faithful family of tracial states and not admitting a continuous character. Then there does not exist a comultiplication $\Delta:A\to{A\tens{A}}$ such that $(A,\Delta)$ is a compact quantum group.
\end{thm}

It follows that the Bratteli-Elliot-Evans-Kishimoto quantum spheres (\cite{beek1,beek2,beek3,ld}) do not admit quantum group structure. The algebra $C_\theta$ of continuous functions on the BEEK two-sphere for the deformation parameter $\theta$ is the fixed point subalgebra of $A_\theta$ under the action of $\ZZ_2$ mapping $u$ and $v$ to their adjoints. This \cst-algebra is nuclear and possesses a faithful trace. Also $C_\theta$ does not admit a character.

The same reasoning now shows that higher dimensional quantum tori (\cite{rie2}) do not admit a compact quantum group either. They do not have characters, have a faithful trace (often unique) and are nuclear. Note also that Theorem \ref{pierw} rules out many $\mathrm{AF}$-algebras because they are nuclear, admit traces and often do not admit characters (like e.g.~all $\mathrm{UHF}$-algebras).

\subsection{Algebras admitting few characters}

In Section \ref{traces} we showed that some \cst-algebras do not admit a compact quantum group structure using the fact that they do not possess a character. However, in some situations, we can use the existence of a character on a \cst-algebra to prove that it cannot be endowed with a comultiplication making it into a compact quantum group. In the cases we consider it is possible to prove that should the quantum space under consideration be a compact quantum group then the Haar measure cannot be a trace. This, together with co-amenability, guarantees existence of many characters which we know is not the case.

In case of the quantum two-spheres of Podle\'s (the standard quantum sphere, see \cite{podles,ld}) and Natsume-Olsen (\cite{natsume,ld}) our reasoning will be based on the following theorem of Bedos, Murphy and Tuset:

\begin{thm}[{\cite[Theorem 2.8]{coam}}]\label{BMT}
Let $\GG=(A,\Delta)$ be a compact quantum group such that its Haar measure is faithful. Then $\GG$ is co-amenable if and only if $A$ admits a character.
\end{thm}

Using this we have:

\begin{thm}\label{AdPod}
The standard Podle\'s quantum sphere $S_{q,0}^2$ is not a compact quantum group.
\end{thm}

\begin{proof}
The \cst-algebra of continuous functions on the standard Podle\'s quantum sphere is isomorphic to $\cK^+$, i.e.~the minimal unitization of the algebra $\cK$ of compact operators on a separable Hilbert space. Assume that $\Delta:\cK^+\to\cK^+\tens\cK^+$ is a comultiplication making $\GG=(\cK^+,\Delta)$ a compact quantum group. Then, first of all, the Haar measure of $\GG$ must be faithful. This is because if it were not faithful, then the \cst-algebra of continuous functions on the reduced compact quantum  group $\GG_r=(A_r,\Delta_r)$ would be one dimensional, and in particular commutative. As we pointed out in Section \ref{pre} commutative $A_r$ means that $\GG$ is co-amenable and $\GG_r=\GG$.

Therefore the Haar measure on $\GG$ is faithful. However $\cK^+$ clearly admits a character, so $\GG$ must be co-amenable by Theorem \ref{BMT}. This means that $\GG$ is at the same time universal, and so $\cK^+$ is the enveloping \cst-algebra of the Hopf $*$-algebra $\cA$ canonically associated to $\GG$. Now let $(f_z)_{z\in\CC}$ be the family of modular functionals on $\cA$ (\cite[Theorem 2.4]{cqg}). Then $(f_{it})_{t\in\RR}$ is a family of $*$-characters of $\cA$, so all these elements must be continuous on the universal completion $\cK^+$ of $\cA$. However there is only one character on $\cK^+$, so $f_{it}=f_0$ for all $t\in\RR$. By the fact that $(f_z)_{z\in\CC}$ is a holomorphic family (\cite[Theorem 2.4(2)]{cqg}) we see that the whole family is trivial, namely $f_z=f_0$ for all $z\in\CC$. By \cite[Theorem 2.5]{cqg} the Haar measure of $\GG$ is a trace, but there is no faithful trace on $\cK^+$.
\end{proof}

We will now show that the Natsume-Olsen quantum spheres cannot be endowed with a compact quantum group structure.

\begin{thm}
The Natsume-Olsen quantum sphere $S^2_t$ is not a compact quantum group for any $t\in[0,\tfrac{1}{2}[$.
\end{thm}

\begin{proof}
The family of Natsume quantum spheres consists really of only two elements. For the deformation parameter $t=0$ we obtain the classical two sphere, which is not a compact group, while for $t\in\bigl]0,\tfrac{1}{2}\bigr[$ the \cst-algebras describing the Natsume quantum sphere are all isomorphic to a fixed \cst-algebra $B$. What we will need now is that $B$ fits into the exact sequence
\[
\xymatrix{0\ar[r]&\cK\tens\C(\TT)\ar[r]&B\ar[r]^\pi&\CC^2\ar[r]&0}
\]
(\cite[Proposition 3.1 \& Corollary 3.2]{natsume}). If we assume that there exists $\Delta:B\to{B\tens{B}}$ such that $\GG=(B,\Delta)$ is a compact quantum group then the Haar measure $h$ of $\GG$ cannot be a trace. Indeed, note that for any $x$ in the ideal $\cI=\cK\tens\C(\TT)=\ker{\pi}\subset{B}$ and any trace $\tau$ on $B$ we have $\tau(x)=0$.\footnote{For any simple tensors $k_1\tens{f_1},k_2\tens{f_2}\in\cK\tens\C(\TT)$ we have $\tau(k_1k_2\tens{f_1f_2})=\tau(k_2k_1\tens{f_2f_1})=\tau(k_2k_1\tens{f_1f_2})$. Therefore $\tau(k\tens{f})=0$ for any $f\in\C(\TT)$ and any $k\in\cK$ which is a finite sum of commutators. From \cite[Theorem 1]{cptkom} we know that these are all compact operators, so $\tau=0$ on $\cK\tens\C(\TT)$. Equivalently we can use the simple fact already used in \cite{natsume} that $\cK\tens\C(\TT)$ is the crossed product $\C_0(\RR)\rtimes_\alpha\ZZ$, where $\alpha$ is translation by one. Clearly there is no invariant probability measure on $\RR$ for this action.}
In particular if $h$ were a trace then $h(x^*x)=0$ for all $x\in\cI$. Therefore if $\GG_r=(B_r,\Delta_r)$ is the reduced quantum group, the ideal $\cJ$ such that $B_r=B/\cJ$ (i.e.~the left kernel of $h$) contains $\cI$. But this means that $B_r$ is a quotient of $B/\cI=\CC^2$ which is commutative. As we argued in the proof of Theorem \ref{AdPod}, this is not possible since compact quantum groups with commutative \cst-algebras are universal.

We will now show that $B_r$ possesses a character. Recall (\cite[Definition 1.7 \& Proposition 3.5]{natsume}) that $B$ is generated by two elements $z$ and $\zeta$ satisfying the following relations:
\begin{eqnarray}
\zeta^*\zeta+z^2=\I=\zeta\zeta^*+(t\zeta\zeta^*+z)^2,\nonumber\\
\zeta{z}-z\zeta=t\zeta(\I-z^2),\qquad\;\quad\label{zeta}
\end{eqnarray}
where $t$ is a fixed number from $]0,\tfrac{1}{2}[$. Note that $z$ cannot belong to any proper ideal of $B$. Indeed, it follows from \eqref{zeta} that 
\[
\zeta=\tfrac{1}{t}\bigl(\zeta{z}-z\zeta+tz^2\bigr)
\]
belongs to any ideal to which $z$ belongs and thus any such ideal must be $B$. Therefore $z$ does not belong neither to $\cI$ nor to $\cJ$.

The case that $B=\cI+\cJ$ must be ruled out because we would then have by \cite[Corollary 1.8.4]{dix}
\begin{equation}\label{IJ}
B_r=B/\cI=(\cI+\cJ)/\cI=\cI/(\cI\cap\cJ)\cong\cK\tens\C(X)
\end{equation}
where $X$ is a compact space (closed subset of $\TT$ because $\cI\cap\cJ$ is an ideal in $\cI=\cK\tens\C(\TT)$, so it must be of the form $\cK\tens\cR$, where $\cR$ is an ideal of $\C(\TT)$). This, however, cannot happen because the last algebra on the right hand side of \eqref{IJ} does not have a unit.

Therefore $z\not\in(\cI+\cJ)$ and so, denoting by $\lambda$ denotes the quotient map $B\to{B_r}=B/\cJ$, we have that $\lambda(\cI)$ is a proper ideal of $B_r$. The (nonzero) quotient
\[
B_r/\lambda(\cI)
\]
is then isomorphic to $B/(\cI+\cJ)$ which in turn is isomorphic to a quotient of $B/\cI=\CC^2$. This algebra clearly has at least one character.

Since $B_r$ has a character, we have by Theorem \ref{BMT} that $\GG_r$ is co-amenable and consequently $\lambda$ is an isomorphism and $\GG=\GG_r=\GG_u$. In particular the Woronowicz characters $(f_{it})_{t\in\RR}$ of $\GG$ must be continuous on $B$. However, since $h$ is not a trace, the family $(f_{it})_{t\in\RR}$ must be a nontrivial continuous family of characters. This stands in contradiction with the fact that the space of characters of $B$ is a discrete two point space.
\end{proof}

\end{document}